%% file: Universality.tex
\documentclass[11pt]{amsart}
\usepackage[pdftex]{graphicx} 
\usepackage{color}
\usepackage[T1]{fontenc}
\usepackage{lmodern}
\usepackage[english]{babel} 
\usepackage[utf8]{inputenc}

\usepackage[centering,margin=2.5cm]{geometry}
\usepackage{upgreek}

\usepackage{amsfonts}
\usepackage{verbatim}
\usepackage{amsmath}
\usepackage[all, cmtip]{xy}\usepackage{amssymb}
\usepackage{amsthm}
\usepackage{extpfeil}
\usepackage{bbm}
\usepackage{paralist}
\usepackage[T1]{fontenc}
\usepackage[colorlinks,citecolor=magenta,urlcolor=magenta,pdftex]{hyperref}

\usepackage{accents}
\usepackage{etex}

\usepackage{pgf}
\usepackage{tikz}
\usetikzlibrary{arrows,automata}
\usetikzlibrary{positioning}

\usepackage{amscd}

\usepackage{nicefrac}
\usepackage{microtype}

\usepackage{mathrsfs}
\usepackage[font=small,labelfont=bf]{caption}
\usepackage[labelformat=simple]{subcaption}

\makeatletter
\newtheorem*{rep@theorem}{\rep@title}
\newcommand{\newreptheorem}[2]{%
\newenvironment{rep#1}[1]{%
 \def\rep@title{#2~\ref{##1}}%
 \begin{rep@theorem}}%
 {\end{rep@theorem}}}
 
\makeatother

\theoremstyle{plain}
\newtheorem*{thm*}{Theorem}
\newtheorem{thm}{Theorem}[section]

\newreptheorem{mthm}{Theorem}
\newreptheorem{mcor}{Corollary}
\newreptheorem{mlem}{Lemma}
\newreptheorem{thm}{Theorem}
\newreptheorem{cor}{Corollary}

\newtheorem{cor}[thm]{Corollary}
\newtheorem{lem}[thm]{Lemma}
\newtheorem*{lem*}{Lemma}

\newtheorem{conj}[thm]{Conjecture}

\theoremstyle{definition}
\newtheorem{dfn}[thm]{Definition}

\newtheorem{rmk}[thm]{Remark}

\newcommand{\htpy}{\sim}
\newcommand{\rsom}{\cR_{\text{om}}}%
\newcommand{\rspol}{\cR_{\text{pol}}}%
\newcommand{\rsins}{\cR_{\text{ins}}}%
\newcommand{\rsdel}{\cR_{\text{del}}}%
\newcommand{\rsomins}{\cR_{\text{om,ins}}}%

\newcommand{\rssphere}{\cS}%

\newcommand{\GA}[1]{\operatorname{Aff}({#1})} %
\newcommand{\GL}[1]{\operatorname{GL}({#1})} %
\newcommand{\PGL}[1]{\operatorname{PGL}({#1})} %
\newcommand{\Mob}[1]{\operatorname{M\ddot{o}b}({#1})} %
\newcommand{\Sim}[1]{\operatorname{Sim}({#1})} %
\newcommand{\Ort}[1]{\operatorname{O}({#1})} %

\newcommand{\NP}{\mathbf{N}}%
\newcommand{\Del}{\cD}
\newcommand{\Sp}[1]{\SS^{#1}}%
\newcommand{\Bll}[1]{\BB^{#1}}%
\newcommand{\sterproj}{\phi}%

\newcommand{\homog}[1]{\operatorname{hom}({#1})}
\DeclareMathOperator{\signe}{sign}

\newcommand{\Lw}{\Lambda}   %

\newcommand{\wh}{\widehat}

\newlength{\dhatheight}
\newcommand{\wwh}[1]{%
\settoheight{\dhatheight}{\ensuremath{\widehat{#1}}}%
\addtolength{\dhatheight}{-0.4ex}%
\widehat{\smash{\widehat{#1}}\phantom{\rule{3pt}{\dhatheight}}}}

\def\cD{\mathcal{D}}

\def\cR{\mathcal{R}}
\def\cS{\mathcal{S}}
\def\cT{\mathcal{T}}

\def\BB{\mathbb{B}}

\def\QQ{\mathbb{Q}}
\def\RR{\mathbb{R}}

\def\RR{\mathbb{R}}
\def\SS{\mathbb{S}}
\def\ZZ{\mathbb{Z}}

\newcommand{\vanish}[1]{}

\def\wh{\widehat}

\def\oli{\overline}
\def\uli{\underline}

\def\({\left(}
\def\){\right)}
\def\no={\,{\,|\!\!\!\!\!=\,\,}}

\def\no={\,{\,|\!\!\!\!\!=\,\,}}

\def\conv{\text{\rm{conv}}}
\def\cone{\text{\rm{pos}}}

\DeclareMathOperator{\verts}{vert}

\def\id{\rm id}

\newcommand{\xqedhere}[2]{%
  \rlap{\hbox to#1{\hfil\llap{\ensuremath{#2}}}}}

\newcommand\Defn[1]{\emph{#1}}
\newcommand\defn[1]{\Defn{#1}}
\newcommand{\cm}[1]{}

\newcommand\mr[1]{\mathrm{#1}}
\newcommand\ol[1]{\overline{#1}}

\newcommand{\bigslant}[2]{{\raisebox{.3em}{$#1$} \Big/ \raisebox{-.3em}{$#2$}}}

\DeclareMathOperator\pos{pos}

\newcommand{\set}[2]{\ensuremath{\left\{#1\,\middle|\,#2\right\}}} %

\title[Universality theorems for inscribed polytopes and Delaunay triangulations]{Universality theorems for inscribed polytopes\\ and Delaunay triangulations}
\author{Karim A.~Adiprasito}
\author{Arnau Padrol}
\author{Louis Theran}
\address{Institut des Hautes \'Etudes Scientifiques, Bures-sur-Yvette, France}
\email{adiprasito@math.fu-berlin.de, adiprasito@ihes.fr}
\address{Institut f\"ur Mathematik, %
Freie Universit\"at Berlin, %
Germany}
\email{arnau.padrol@fu-berlin.de, theran@math.fu-berlin.de}

\date{}
\thanks{K.~A.~Adiprasito acknowledges support by an EPDI postdoctoral fellowship and by the Romanian NASR,
CNCS --- UEFISCDI, project PN-II-ID-PCE-2011-3-0533. The research of A.~Padrol is supported by the DFG
Collaborative Research Center SFB/TR~109 ``Discretization
in Geometry and Dynamics''.  The research of L. Theran supported by 
the European Research Council under the European Union's Seventh
Framework Programme (FP7/2007-2013) / ERC grant agreement no 247029-SDModels.
A preliminary version of some of the results of this paper has appeared in~\cite{PadrolTheran14}.}

\begin{document}
\begin{abstract}
We prove that every primary basic semialgebraic set is homotopy equivalent to the set of inscribed realizations (up to M\"obius transformation) of a polytope. If the semialgebraic set is moreover open, then, in addition, we prove that (up to homotopy) it is a retract of the realization space of some inscribed neighborly (and simplicial) polytope. We also show that all algebraic extensions of $\QQ$ are needed to coordinatize inscribed polytopes. 
These statements show that inscribed polytopes exhibit the Mn\"ev universality phenomenon. 

Via stereographic projections, these theorems have a direct translation to universality theorems for Delaunay subdivisions. In particular, our results imply that the realizability problem for Delaunay triangulations is polynomially equivalent to the existential theory of the reals.

\end{abstract}

\maketitle

\input{intro}
\input{preliminaries}

\input{inscribed}

\input{simp_inscribed}

\bibliographystyle{amsalpha}
\bibliography{universality}
\appendix
\end{document}

%% file: intro.tex
\section{Introduction}
The Delaunay subdivision of a set of points in~$\RR^d$ plays a central role in 
computational geometry~\cite{Edelsbrunner2006}.  A few applications are:
nearest-neighbor search, pattern matching, clustering, and mesh generation.
Via stereographic projection, Delaunay subdivisions can be lifted to \emph{inscribed polytopes} 
\cite{Brown1979}---those with all vertices on the unit sphere---in one dimension higher, 
so that Delaunay triangulations lift to simplicial inscribed polytopes. The study of
inscribed polytopes, and in particular the problem of deciding whether a 
polytope admits an inscribed realization, is a classical subject 
\cite{Steiner1832}\cite{Steinitz1928}\cite{Rivin1994} in which many
fundamental questions are still open~\cite{GonskaZiegler2011}.

In this paper, we are interested in \emph{realization spaces} of a \emph{fixed 
combinatorial type} of Delaunay 
subdivision/inscribed polytope. For a configuration $A$ of $n$ points in 
$\RR^d$, which we assume to be labeled by $[n]=\{1,\dots,n\}$, the cells of its 
Delaunay subdivision are represented by a family $T$ of subsets of $[n]$. The 
\emph{realization space} $\rsdel({T})$ is a parametrization of the set of 
all configurations of $n$ labeled points whose Delaunay triangulation has the 
combinatorial structure of $T$ (as a polytopal complex with vertex set $[n]$).

Analogously, $\rsins({P})$, the realization space of an inscribed polytope $P$, 
is a parametrization of configurations of $n$ points in the unit sphere whose 
convex hull has the same face lattice as $P$.

In dimension $2$, results of~\cite{Rivin1994} imply that both of 
these realization spaces are homeomorphic to a polytope (that 
depends on $T$ only).  This completely determines the topological 
structure of both realization spaces.  For example, they are
always \emph{connected} and \emph{contractible}.

\subsection{Universality for Delaunay subdivisions}
Our main results show that, in higher dimensions, 
the situation is completely different: the realization space 
of a $d$-dimensional Delaunay subdivision (or triangulation) can be 
arbitrarily complicated.%
\begin{repthm}{thm:lawrence}
For every primary basic semi-algebraic set there is a Delaunay subdivision and an 
inscribed polytope whose realization space is homotopy equivalent to $S$.
\end{repthm}
Said differently, realization spaces of inscribed polytopes 
exhibit the \emph{topological universality} in the sense of Mnëv \cite{Mnev1988}.  We 
also show that these realization spaces exhibit \emph{algebraic universality} (also 
a notion from \cite{Mnev1988}).

\begin{repcor}{cor:lawrenceAlg}
For every finite field extension $F/\QQ$ of the rationals, there is a realizable
Delaunay subdivision (equivalently, an inscribed polytope) 
that cannot be realized with coordinates in $F$.
\end{repcor}

\subsection{Universality for Delaunay triangulations}
The subdivisions constructed in proof of Theorem \ref{thm:lawrence} are far from being 
triangulations. To insist on triangulations and simplicial polytopes requires 
different tools. We adapt a recent proof of the Universality Theorem 
for simplicial polytopes \cite{AdiprasitoPadrol2014} to obtain a weak universality 
theorem for Delaunay triangulations. 

\begin{repthm}{thm:neighborly}
For every open primary basic semi-algebraic set $S$ there is a (neighborly) Delaunay triangulation and an 
inscribed simplicial (neighborly) polytope,
such that $S$ is a retract of their realization spaces, up to homotopy equivalence.
\end{repthm}

\subsection{Complexity}
The complete statements of these theorems provide
linear bounds for the number points of the triangulation and the dimension in terms
of the \emph{arithmetic complexity} of the corresponding semi-algebraic sets.
Moreover, $S$ is non-empty \emph{if and only if} $\rsdel(T)$ is. Such a triangulation 
can be computed from $S$ in polynomial time, which shows that deciding whether a Delaunay
triangulation is realizable is hard. Indeed, the proof of Theorem~\ref{thm:neighborly} 
shows that deciding realizability of Delaunay triangulations is as hard as deciding
realizability of rank~$3$ oriented matroids~\cite{Mnev1988}\cite{Shor1991}.

\begin{repcor}{cor:hard}
The {realizability problem for Delaunay triangulations and inscribed simplcial 
polytopes} is polynomially equivalent to the \emph{existential theory of the 
reals (ETR)}. In particular, it is NP-hard.
\end{repcor}

Another consequence of this effective bound is that the number of 
connected components of the realization space of a $d$-dimensional Delaunay 
triangulation can be exponential in~$d$.
\begin{repcor}{cor:asymptotic}
For every $m\geq 1$ there exist configurations of $O(m)$ points in general 
position in $\RR^{O(m)}$ whose 
realization spaces as Delaunay triangulations have at least $2^m$ connected 
components.
\end{repcor}

Our smallest example of a triangulation with disconnected realization space is 
in $\RR^{25}$, which leaves open the existence of these configurations in $\RR^d$ for each 
$3\leq d\leq 24$.
\begin{repcor}{cor:small-config}
There is a $25$-dimensional configuration of $30$ points whose Delaunay 
triangulation has 
a disconnected realization space.
\end{repcor}

\subsection{Context and related work}
The realization spaces of $3$-dimensional inscribed polyhedra are well 
understood.  On the other hand, there is a rich theory of the ``wildness'' of 
realization spaces of higher-dimensional polyhedra can be quite wild.  Here is 
the background and connection with our results.

\subsubsection{Dimension $2$ and inscribable polyhedra}  
Theorems \ref{thm:lawrence} and \ref{thm:neighborly} and their corollaries 
should be contrasted with fundamental results of Rivin 
\cite{Rivin1994}\cite{Rivin1996}\cite{Rivin2003} that connect $2$-dimensional 
Delaunay subdivisions with metric properties of hyperbolic $3$-dimensional 
polyhedra.

Rivin's work in particular entailed that: (1) whether a (combinatorial) planar 
graph has a drawing as a Delaunay triangulation can be tested in polynomial 
time; (2) that the realization space of a planar Delaunay triangulation is 
homeomorphic to a polyhedron of so-called \emph{angle structures} (see 
\cite{FuterGueritaud11} for an elementary introduction to the method), and, in 
particular, connected.

In the language of polyhedra, (1) says that whether a graph is the $1$-skeleton of an 
inscribable polyhedron is efficiently 
checkable; and (2) says that the set of inscribed realizations is convex (and in particular 
contractible) in the parameterization by dihedral angles.  

The question of whether every polyhedron is inscribable had been first raised  
by Steiner in 1832~\cite{Steiner1832}, with the first negative examples 
given by Steinitz in 1928~\cite{Steinitz1928}. This makes such a sharp 
characterization 
of the inscribable types and their realization spaces a surprising breakthrough. 
In contrast, Theorems~\ref{thm:lawrence} and \ref{thm:neighborly} 
suggest that a polynomial time characterization for all dimensions is, under 
standard conjectures, not possible.

\subsubsection{Higher dimensions and universality}
A general principle in the theory of realization spaces for (semi-)algebraically defined 
objects is succinctly put in \cite{Vakil06}:  ``Unless there is some a priori reason otherwise, 
the deformation space may be as bad as possible.''

Underlying a large number of these kinds of phenomena is a paradigmatic result 
of Mn\"ev. The Universality Theorem states that for every [open] primary basic 
semi-algebraic set there is a [uniform] oriented matroid of rank~$3$ whose 
realization space is stably equivalent to it. (The survey 
\cite{RichterGebert1998} provides an accessible presentation of this and related 
results and their proofs, and a more computationally oriented approach can be 
found in \cite{Shor1991}.)

The Universality Theorem in particular entails a negative answer to Ringel's 1956 
\emph{isotopy problem}, which asked whether, given two point configurations $A_0$ and $A_1$ 
with the same oriented matroid (order type), is it always possible to find a continuous path 
of point configurations $\{A_t\}_{0\leq t\leq 1}$ with the same oriented matroid?  (This 
weaker result also follows via examples from
\cite{JaggiManiLevitskaSturmfelsWhite1989}
\cite{RichterGebert1996}
\cite{Suvorov1988}
\cite{Tsukamoto2013}
\cite{Vershik1988}
\cite{White1989}.)
Actually, the Universality Theorem shows that there are oriented matroids 
that have realization spaces with arbitrarily many components.  

Another straightforward consequence of the Universality Theorem is that 
determining realizability 
of oriented matroids is polynomially equivalent to the \emph{existential theory 
of the reals}, and in particular NP-hard \cite{Mnev1988}\cite{Shor1991}.

Via a reduction given in \cite{Mnev1988}, realization spaces of polytopes also 
exhibit universality: for every semi-algebraic variety $S\subset \RR^s$, there 
is a polytope $P$ whose realization space is stably equivalent to~$S$. Here, the 
realization space of a polytope $P$ is the set of point configurations whose 
convex hull is combinatorially equivalent to $P$. In principle, this polytope 
might be of a very high dimension. However, 
Richter-Gebert~\cite{RichterGebert1997} made a breakthrough when he proved that 
there is universality already in realization spaces of $4$-dimensional 
polytopes.  Again, there is a contrast with $3$-polytopes, which have 
contractible realization spaces (see, e.g.,~\cite[Part IV]{RichterGebert1997}).

When the semi-algebraic sets are open, one can furthermore require the polytopes
to be simplicial (and even neighborly), although only in arbitrarily high 
dimensions~\cite{AdiprasitoPadrol2014}\cite{Mnev1988}. 
Universality for simplicial polytopes in fixed dimension remains wide open.

However, the existence of Delaunay triangulations with a disconnected 
realization space is not a direct consequence of the results of Richter-Gebert 
and Mn\"ev. Indeed, even if $P$ is a polytope with a disconnected realization 
space, it could be that the variety $\rssphere$ that certifies that all the 
vertices lie on the unit sphere does not intersect with all its connected 
components---or any of them. Hence, to present a 
Delaunay triangulation with a disconnected realization space, one has to show 
that $\rssphere$ hits at least two of these connected components or that the 
intersection of $\rssphere$ with a connected component is disconnected.

On the other hand, some universality phenomena from the theory of general polytopes
are already known to carry over to the case of inscribed polytopes; for instance, there are infinitely many 
projectively unique {inscribed} polytopes even in \emph{bounded} dimension, and 
every inscribed polytope is the face of some projectively unique inscribed polytope, 
cf.\ \cite{AdiprasitoZiegler2014}.

\subsection{Open problem: universality in fixed dimension}

Theorems~\ref{thm:lawrence} and \ref{thm:neighborly} are the first step towards 
a universality theory for Delaunay triangulations and leave several open questions. 
First of all, a strong version of Theorem~\ref{thm:neighborly} should state homotopy equivalence
between the realization space of the Delaunay triangulation and the semi-algebraic set.

The main challenge is to prove a Universality Theorem for Delaunay Triangulations in fixed dimension.
Recall that polytopes present universality already in dimension~$4$ (for simplicial 
polytopes this is also conjectured). Since the results here run more or less in parallel
with the development of the theory for polytopes, the strongest conjecture we can make is:
\begin{conj}\label{conj:universality}
For every [open] primary basic semi-algebraic set $S$ defined over $\ZZ$
there is a $3$-dimensional Delaunay [triangulation] subdivision whose realization space 
is homotopy equivalent to~$S$. 
\end{conj}
 
Since connected sums, the main ingredient of Richter-Gebert's proof of the Universality Theorem for
$4$-polytopes~\cite{RichterGebert1997}, do not behave well with respect to inscribability, it seems that 
a new set of tools will be needed to prove our conjecture.

\subsection{Reading guide}
The rest of this paper is organized as follows: 
Section~\ref{sec:preliminaries} introduces some necessary notation. 
Section~\ref{sec:inscribed} is devoted to the Universality Theorem for inscribed
polytopes and Delaunay subdivisions. Simplicial polytopes and triangulations
require different tools, and are studied in Section~\ref{sec:simpinscribed}.

%% file: preliminaries.tex
\section{Preliminaries}\label{sec:preliminaries}
\subsection{Notation}
For a quick reference for oriented matroids, polytopes and Delaunay triangulations,
we refer to the chapters \cite{RichterGebertZiegler1997}, \cite{HenkRichterGebertZiegler1997}
and \cite{Fortune1997} of the handbook~\cite{HanbookDCG}, respectively. Our notation
coincides mostly with theirs.

Let $V$ be a configuration of $n$ vectors in $\RR^r$, which are labeled by 
elements in $[n]=\{1,\dots,n\}$. Consider the map $\chi^V:[n]^{r}\mapsto 
\{0,1,-1\}$ that for each tuple $(i_1,\dots,i_{r})$ assigns the sign 
\[\chi^V(i_1,\dots,i_{r})=\signe\det(v_{i_1},\dots,v_{i_{r}}).\]
The map $\chi^V$ is called the \defn{chirotope} of $V$ and determines its 
\defn{oriented matroid}, which has \defn{rank}~$r$ 
(see~\cite{OrientedMatroids1993} for a comprehensive introduction to oriented 
matroids).

Now, let $A$ be a configuration of $n$ points in $\RR^d$. The \defn{homogenization} of 
$A$ is a vector configuration $\homog{A}=\{\bar a_1,\dots,\bar 
a_n\}\subset\RR^{d+1}$ obtained by appending $1$ as the last coordinate of the 
points of $A$: $\bar a_i=(a_i,1)$.
The \defn{oriented matroid} of $A$ is defined to be the oriented matroid of its 
homogenization~$\homog{A}$. 

The oriented matroid of a point configuration is always \defn{acyclic}.
A point configuration $A$ is in \defn{general position} if no $d+1$ points of 
$A$ lie in a common hyperplane, and then it defines a \defn{uniform} matroid. 

The convex hull of $A$ is a \defn{polytope} $P=\conv(A)\subset\RR^d$ and the 
intersection of $P$ with a supporting hyperplane is a \defn{face} of $P$. Faces 
of dimensions $0$ and $d-1$ are called \defn{vertices} and \defn{facets}, 
respectively. A point configuration $A$ is in \defn{convex position} if it 
coincides with $\verts(P)$, the set of vertices of $P=\conv(A)$. If $A$ is in 
convex position, each face $F$ of $P$ can be identified with the set of labels 
$\set{i\in[n]}{a_i\in F}$. The \defn{face lattice} of $P$ is then a poset of 
subsets of $[n]$. In this context, two vertex-labeled polytopes are 
\defn{combinatorially equivalent}, denoted $P\simeq Q$, if their face lattices 
coincide. We call $P$ \defn{inscribed} if all its vertices lie in the unit 
sphere~$\Sp{d-1}$, and \defn{inscribable} if it is combinatorially equivalent to 
an inscribed polytope.

The face lattice of a polytope $P$ coincides with that of the \defn{convex cone} obtained
as the positive hull of its homogenization 
$\cone(\hom(P)):=\set{\sum \lambda_i x_i}{\lambda_i \geq 0, \, x_i\in \hom(P)}$. 

The oriented matroid of a polytope $P$ is \defn{rigid} if its face lattice of 
determines the oriented matroid of its set of vertices (see \cite[Section 
6.6]{Ziegler1995}). In the language of the next section, a polytope $P$ is rigid 
if and only if $\rsom(\verts({P}))= \rspol(P)$.

A \defn{subdivision} of a point configuration $A$ is a collection $\cT$ of 
polytopes with vertices in $A$, which we call \defn{cells}, that cover the 
convex hull of $A$ and such that any pair of polytopes of $\cT$ intersect in a 
common face. A \defn{triangulation} of $A$ is a subdivision where all the cells 
are simplices. Again, a subdivision of $A$ can be identified with a poset of 
subsets of~$[n]$. Two subdivisions $\cT$ and $\cT'$ of two labeled 
configurations $A$ and $A'$ are \defn{combinatorially equivalent}, denoted by 
$\cT\simeq\cT'$, if their respective posets coincide.

The \defn{Delaunay subdivision} $\Del(A)$ of a point configuration 
$A\subset\RR^d$ is the subdivision that consists of all cells defined by the 
\defn{empty circumsphere condition}: $S \in \Del(A)$ if and only if there exists 
a $(d-1)$-sphere that passes through all the vertices of $S$ and all other 
points of $A$ lie outside this sphere. If $A$ is in general position and no 
$d+2$ points of $A$ lie on a common sphere, then the empty circumsphere 
condition always defines a simplex of~$A$, and hence the Delaunay subdivision is 
a triangulation,  the \defn{Delaunay triangulation} of $A$.

We denote \defn{homeomorphic} sets $S$ and $T$ by $S\cong T$ and \defn{homotopic}
sets by $S\htpy T$ (see \cite[Section~58]{Munkres} for 
definitions).  We also recall that a continuous map $f : S\to T$ is a  \defn{retraction} 
of $S$ onto $T$ if there
is a continuous map $g:T\to S$ such that $f\circ g=\id$. 
If a retraction exists, then $T$ is a \defn{retract} of~$S$. If moreover $g \circ f$ is homotopic
to the identity, then $T$ is a \defn{deformation retract} of $S$, and $T\htpy S$.

\subsection{Realization spaces}\label{sec:realizationspaces}

We will work with the following \defn{realization spaces}. 
Observe that for oriented matroids and polytopes, we work with the 
acyclic vector configurations arising from homogenization. (This approach is 
convenient for technical reasons, and used often, for example in~\cite{OrientedMatroids1993}.)
We also identify a $d$-dimensional configuration of $n$ points or vectors with the 
corresponding tuple in $\RR^{d\times n}$ containing the 
coordinates, ordered according to their labels.

\begin{itemize}[$\bullet$]
\item The \defn{realization space of an oriented matroid} $M$ (of rank $d+1$ with
$n$ elements), that we denote $\rsom(M)\subset \RR^{(d+1)\times n}$, is the set of vector
configurations that realize $M$, up to linear transformation:
\[\rsom(M)=\bigslant{\set{V\in \RR^{(d+1)\times n}}{V \text{ realizes }M}}{\GL{\RR^{d+1}}}.\]
\item The \defn{realization space of a polytope} $P$
(with $n$ vertices in~$\RR^d$), that we denote
$\rspol(P)\subset \RR^{(d+1)\times n}$, is the set of acyclic configurations whose positive span 
is combinatorially equivalent to the cone over $P$, up to linear transformation:
\[\rspol(P)=\bigslant{\set{V\in \RR^{(d+1)\times n}}{\cone(V)\simeq P}}{\GL{\RR^{d+1}}}.\]
\item The \defn{realization space of an inscribed polytope} $P$ 
(with $n$ vertices in~$\RR^d$), that we denote
$\rsins(P)\subset (\Sp{d-1})^n \subset \RR^{d\times n}$, is the set of inscribed 
point configurations whose convex hull
is combinatorially equivalent to $P$, up to M\"obius transformation:
\[\rsins(P)=\bigslant{\set{A\in (\Sp{d-1})^n}{\conv(A)\simeq P}}{\Mob{\Sp{d-1}}}.\]
\item The \defn{realization space of a Delaunay subdivision} $T$ (of $n$ points 
in $\RR^d$), that we
denote $\rsdel(T)\subset \RR^{d\times n}$, is the set of point configurations whose
Delaunay triangulation is combinatorially equivalent to $T$, 
up to similarity:
\[\rsdel(T)=\bigslant{\big\{ A\in \RR^{dn} \,\big|\, \Del(A)\simeq T\}}{\Sim{\RR^d}}.
\]\end{itemize}

For a given point configuration $A$, we abuse notation and use $\rsom(A)$, 
$\rspol(A)$, $\rsins(A)$, 
and $\rsdel(A)$
to denote the realization spaces $\rsom(\chi^A)$, $\rspol(\conv(A))$, 
$\rsins(\conv(A))$ and $\rsdel(\Del(A))$, respectively.

\begin{rmk}
A number of alternative definitions are possible. For example, 
by factoring different transformation groups or by considering non-homogenized configurations.
Most of these definitions are actually homotopy-equivalent, as we discuss below, and hence our results hold anyway.
These definitions are natural because the groups preserve spheres.

We also leave out the combinatorics of the boundary for Delaunay subdivisions, 
which amounts to take also into account the empty spheres that go through the ``point at infinity''.
Although these two definitions are not necessarily homotopy-equivalent, again, 
our results hold for both kinds of definition, see also Remark~\ref{rmk:fixedboundary}.
\end{rmk}

It is sometimes useful to commute between different realization spaces; we state the straightforward lemmata, without detailed proof, here:
\begin{lem}[Realization spaces of matroids]\label{lem:homogenization}
Let $M$ be an acyclic oriented matroid.  Then the following three %
spaces are homotopy equivalent:
\begin{compactitem}[$\circ$]
\item The realization space of homogeneous configurations, modulo linear transformations:
\begin{equation}\label{eq:lin-relspace}
	\rsom(M)%
	=\bigslant{\set{V\in \RR^{(d+1)\times n}}{V \text{ realizes }M}}{\GL{\RR^{d+1}}}
\end{equation}
\item  The realization space of affine configurations, modulo admissible projective transformations:
\begin{equation}\label{eq:proj-relspace}
	\rsom^{\textsf{proj}}(M)=\bigslant{\set{V\in \RR^{d\times n}}{V \text{ realizes }M}}{\PGL{\RR^{d}}}
\end{equation}
\item The realization space of affine configurations, modulo affine transformations:
\begin{equation}\label{eq:aff-relspace}
	\rsom^{\textsf{aff}}(M)=\bigslant{\set{A\in \RR^{d\times n}}{A \text{ realizes }M}}{\GA{\RR^{d}}}
\end{equation}
\end{compactitem}
\end{lem}

\begin{proof}[Proof (sketch)]
To see \eqref{eq:lin-relspace}$\htpy$\eqref{eq:proj-relspace}, consider the 
map $\rsom^{\textsf{lin}}(M)\to\rsom^{\textsf{proj}}(M)$
that sends a vector configuration $V$ to the intersection of its positive span
with a hyperplane that intersects every positive ray spanned by $V$. 
The map is well defined because two point configurations arising from different hyperplanes 
are related by an admissible projective transformation, 
and all linear transformations of $V$ also induce admissible projective transformations of $A$. 
The homogenization map provides a section, and the fibers are easily seen to be contractible.

For \eqref{eq:proj-relspace}$\htpy$\eqref{eq:aff-relspace}, observe that each fiber of the quotient map
$\rsom^{\textsf{aff}}(M)\to \rsom^{\textsf{proj}}(M)$ is homeomorphic to the 
set of admissible projective transformations up to affine transformation. That is,
the set of ``hyperplanes at infinity'' that do not cut $\conv(A)$. This, in turn, is homeomorphic to a polytope, the polar polytope
of $\conv(A)$, which depends continuously on $A$. Hence, a continuous section can be defined by selecting its barycenter.
\end{proof}

Similarly, we have the following lemma for inscribed poltopes:

\begin{lem}[Realization spaces of inscribed polyopes]\label{lem:eqiv_relins}
Let $P$ be an inscribed polytope in $\RR^d$.  Then the following three spaces are homotopy equivalent:
\begin{compactitem}[$\circ$]
\item The realization space of all inscribed polytopes combinatorially equivalent to it, modulo M\"obius transformations:
\begin{equation}
	\rsins(P)%
	=\bigslant{\set{A\in (\Sp{d-1})^n}{\conv(A)\simeq P}}{\Mob{\Sp{d-1}}}.
\end{equation}
\item  The realization space of all inscribed polytopes combinatorially equivalent to it, modulo orthogonal transformations:
\begin{equation}
	\rsins^{\textsf{ort}}(P)=\bigslant{\set{A\in (\Sp{d-1})^n}{\conv(A)\simeq P}}{\Ort{\RR^{d}}}.
\end{equation}
\end{compactitem}
\end{lem}

\subsection{Mn\"ev's universality theorem}\label{sec:universality}

A \defn{primary basic semi-algebraic} set is a subset of $\RR^d$ defined by integer polynomial equations and inequalities \[S=\set{\mathbf{x}\in\RR^d}{f_1(\mathbf{x})=0,\dots,f_k(\mathbf{x})=0, f_{k+1}(\mathbf{x})>0,\dots,f_r(\mathbf{x})>0},\text{  where }f_i\in\ZZ[\mathbf{x}].\]
Realization spaces of polytopes and oriented matroids are examples of primary basic semi-algebraic sets. Mn\"ev's Universality Theorem \cite{Mnev1988} is a reciprocal statement:
every primary basic semi-algebraic set appears as the realization space of some oriented matroid/polytope up to \defn{stable equivalence}, which implies homotopy equivalence 
(see~\cite{RichterGebert1998}). We refer to \cite{Mnev1988}\cite{RichterGebert95}\cite{RichterGebert1998} for its proof.

\begin{thm}[{Universality Theorem~\cite{Mnev1988}}]\label{thm:universality}
For every primary basic semi-algebraic set $S$ defined over $\ZZ$ there is a
rank~$3$ oriented matroid whose realization space is stably equivalent to~$S$. 
If moreover $S$ is open, then the oriented matroid may be chosen to be uniform.

Given any presentation of $S$, such an oriented matroid of size linear in the 
size of the presentation can be found in polynomial time.
In particular, there is such a matroid whose size is linear in the sum of the arithmetic 
complexities of the polynomials.

\end{thm}

The \defn{arithmetic complexity} of a polynomial $f\in\ZZ[x]$ is, roughly speaking, the minimal number of operations $+$ and $\times$ needed to compute it from $x$ and $1$, when we are allowed to reuse computations. For example, $(x + 1)^2=(x + 1)(x + 1)$ can be computed with one addition and one multiplication. See \cite{BuergisserClausenShokollahi1997}\cite{Valiant1979} for details.

The following statements are among the consequences of the Universality Theorem:
\begin{cor}\label{cor:omhard}
The {realizability problem for oriented matroids of rank~$3$} is polynomially equivalent to the \emph{existential theory of the reals (ETR)}.
\end{cor}
\begin{cor}\label{cor:omalgebraic}
For every finite field extension $F/\QQ$ of the rationals, there exists an
oriented matroid of rank~$3$ that cannot be realized with coordinates in $F$. 
\end{cor}
In other words, that ``all algebraic numbers'' are needed to coordinatize oriented matroids.

%% file: inscribed.tex
\section{Universality for inscribed polytopes and Delaunay subdivisions.}\label{sec:inscribed}

In this section, we prove:

\begin{thm}\label{thm:lawrence}
For every primary basic semi-algebraic set there is a Delaunay subdivision and 
an inscribed polytope whose realization space is homotopy equivalent to $S$.
\end{thm}

To pass from 
realization spaces of oriented matroids to those of polytopes, we 
use (as in \cite{Mnev1988}) Lawrence extensions. The resulting 
polytopes are always inscribable, as observed in \cite{AdiprasitoZiegler2014}.

\subsection{Lawrence polytopes}
We recall some properties of polytopes constructed from Lawrence extensions.
\begin{dfn}[cf. \cite{RichterGebert1998}]
Let $A$ be a $d$-dimensional point configuration and let $a\in A$. The 
\Defn{Lawrence extension} of $A$ on $a$ is the $(d+1)$-dimensional point 
configuration 
\[\Lw(A,a):=(A\setminus a)\cup \oli a \cup \uli a,\]
where $A$ is embedded in the hyperplane $x_{d+1}=0$ and the new points are $\uli 
a:=(a,1)$ and $\ol a:=(a,2)$.
Let $B\subseteq A$, the \Defn{Lawrence extension} $\Lw(A,B)$ is the point 
configuration obtained by Lawrence lifting the points of $B$ one by one
\[\Lw(A,B):=\Lw(\Lw(\dots\Lw(\Lw(A,b_1),b_2)\dots b_{k-1}),b_k),\]
where $B=\{b_1,\dots,b_k\}$.

The \Defn{Lawrence polytope} of a point configuration $A$ is the polytope 
$\Lw(A)=\conv(\Lw(A,A))$.
\end{dfn}

\begin{lem}[{{cf. \cite[Theorems 6.26 and 
6.27]{Ziegler1995}}}]\label{lem:Lawrencerigid}
For any point configuration $A$, $\Lw(A,A)$ is in convex position and the 
Lawrence polytope $\conv(\Lw(A,A))$ is rigid.
\end{lem}

\subsection{Partially inscribed point configurations}

Given an oriented matroid $M$ and a subset of its elements $E$, we consider the 
set of realizations of $M$ such that the 
points of $E$ lie on the boundary of the unit ball $\Bll{d}$ and all the remaining points are outside. 
We use the homogenized version and consider 
such realizations up to orthogonal transformations fixing the hyperplane $x_{d+1}=0$.
\begin{align*}
\rsomins(M,E)=
\big\{V\in \RR^{(d+1)\times n}\}\,\big|\,& V \text{ realizes } M \text{, } \forall e\in E,\, A_e\in \partial \pos(\homog{\Bll{d}})  \\
&\bigslant{\text{ and }\forall e\notin E,\, 
A_e\notin \pos(\homog{\Bll{d}}) \big\}}{\Ort{\RR^d}}.
\end{align*}
Notice that with this definition we are implicitly allowing points at infinity, and negative points, 
when we consider vectors that span rays not 
intersecting the homogenizing hyperplane.

The following lemma expands \cite[Proposition~A.5.8]{AdiprasitoZiegler2014} to 
make a statement about realization spaces.
\begin{lem}\label{lem:insLawrence}
For every planar point configuration $A$, $\rsom(A)\htpy \rsins(\Lw(A))$.
\end{lem}
\begin{proof}
Since Lawrence polytopes are rigid, and using Lemma~\ref{lem:eqiv_relins}, we have that $\rsins(\Lw(A))\htpy \rsomins(\Lw(A),\Lw(A))$. 
Therefore, we just need to prove that 
\begin{equation}\label{eq:rsom-rsomins}
	\rsom(A)\htpy \rsomins(\Lw(A),\Lw(A))
\end{equation}
We prove first that for every subset $B\subseteq A\subset 
\RR^d$ and for every $a\in A\setminus B$, 
\begin{equation}\label{eq:risomins-risominsL}
	\rsomins(A,B)\htpy \rsomins(\Lw(A,a),B\cup \{\uli{a},\oli{a}\}).
\end{equation}
For every realization of $\Lw(A,a)$ one can recover a realization of~$A$ by 
intersecting the ray emanating at $\oli{a}$ through $\uli{a}$ with the linear 
hyperplane~$H$ spanned by the remaining points.  For the moment, 
assume~$H$ is an equator of the unit sphere $\Sp{d-1}$.  In 
this case, it is clear we recover a realization of~$A$ with all the points of~$B$
on $\Sp{d-2}$ and all the remaining points outside $\Bll{d-1}$.
In general, $H$ will not be an equator, but then there is a unique rescaling
that sends $H\cap \Sp{d-1}$ to $\Sp{d-2}$.
This map is a well-defined projection 
\begin{equation}\label{eq:lawerenceproj}
	\rsomins(\Lw(A,a),B\cup \{\oli a, \uli a\})\longrightarrow \rsomins(A,B)
\end{equation}
because every orthogonal transformation of $\RR^d$ induces an orthogonal transformation on $H$.  

What's left is to establish that the fibers of this continuous map are non-empty and 
contractible.
First, by reflection symmetry we may assume that $\uli{a}$ and $\oli{a}$
are in the positive half-space defined by~$H$.  With this, we then see that a point
in the fiber is parameterized by the center of the sphere and the location of~$\oli{a}$,
This is the product of a line and a
non-empty (spherically) convex subset of the sphere (points in the upper spherical cap 
visible from $a$), and so contractible. Hence, \eqref{eq:risomins-risominsL}
is established. Finally, we construct a continuous inverse by selecting~$\oli{a}$ to be the barycenter
of the possible locations; since the fibers behave Hausdorff continuous on the pair, we are done.

To get to \eqref{eq:rsom-rsomins}, we will prove that 
$\rsom(A)\sim \rsomins(\Lw(A,a),\{\uli{a},\oli{a}\})$ for any $a\in A$ (and then we 
only need to apply \eqref{eq:risomins-risominsL} to the remainaing points).
Here the fibers of the projection are the set of choices for the 
sphere (the spheres touching the upper half-space not containing any 
point of $A$) product with the choices for $\oli a$ (again, a convex set).
We can factor the projection map through the quotient $\GL{\RR^d}/\Ort{\RR^d}$.
\end{proof}

\subsection{Topological universality}
Now Theorem~\ref{thm:lawrence} follows directly from the combination of the 
Universality Theorem~\ref{thm:universality} with Lemma~\ref{lem:insLawrence}.
\begin{proof}[Proof of Theorem~\ref{thm:lawrence}]
 By the Universality Theorem~\ref{thm:universality}, for every primary basic 
 semi-algebraic set $S$ there is a point configuration whose realization space $\rsom(A)$
 is homotopy equivalent to~$S$. Now, by Lemma~\ref{lem:insLawrence}, $\rsom(A)\htpy \rsins(\Lw(A))$.
 Finally, if we consider the Delaunay subdivision $T$ consisting of a single cell combinatorially equivalent to~$\Lw(A)$,
 one can easily see that $\rsdel(T)\htpy \rsins(\Lw(A))$.
\end{proof}

\subsection{Algebraic universality}
Corollary~\ref{cor:omalgebraic} 
follows at once from the Universality Theorem~\ref{thm:universality} because of stable equivalence.
Although the exact notion of stable equivalence does not hold in our
situation, the statement analogous to Corollary~\ref{cor:omalgebraic}
does.

\begin{cor}\label{cor:lawrenceAlg}
For every finite field extension $F/\QQ$ of the rationals, there is a realizable
Delaunay subdivision (equivalently, an inscribed polytope) that cannot be realized with coordinates in $F$. 
\end{cor}
\begin{proof}
 By Corollary~\ref{cor:omalgebraic}, for every algebraic extension $F$ of the 
rational numbers, there is a point configuration $A$ that cannot be 
coordinatized in $F$. Now, by Lemma~\ref{lem:insLawrence}, the Lawrence polytope 
$\Lw(A)$ is inscribable. Any inscribed realization of $\Lw(A)$ encodes a 
realization of $A$, which can be obtained through a series of radial projections
(see the proof of Lemma~\ref{lem:insLawrence}). 
Hence, if $\Lw(A)$ had a realization with coordinates in $F$, so would~$A$.
\end{proof}

%% file: simp_inscribed.tex
\section{Universality for inscribed simplicial polytopes and Delaunay triangulations}
\label{sec:simpinscribed}
To obtain universality results for simplicial polytopes and triangulations, we 
cannot use Lawrence extensions, which produce configurations with a lot of 
non-simplicial faces. Instead, we will use \emph{neighborly polytopes}, which are also rigid.
This is possible, by a result of Kortenkamp, which implies that we can embed the oriented 
matroids of a planar point configuraitons inside the oriented matroid of a neighborly polytope.

\subsection{Stereographic projections}
The \defn{stereographic projection} $\sterproj:\Sp{d}\setminus \NP\subset 
\RR^{d+1}\rightarrow \RR^{d}$ is the map defined by 
\[\sterproj(x_1,\dots,x_{d+1})= 
\left(\frac{x_1}{1-x_{d+1}},\dots,\frac{x_d}{1-x_{d+1}}\right),\]
where $\NP$ is the north pole of the unit sphere $\Sp{d}$. 

The sterographic projection and its inverse are classical tools to translate 
from Delaunay triangulations to inscribed polytopes, and vice 
versa~\cite{Brown1979}.
The following lemma explains how to relate realizations of the Delaunay triangulations
and inscribed realizations of  polytopes.

\begin{lem}\label{lem:stereographic}
$A=\{a_1,\dots,a_{n}\}$ be a configuration of $n$ points in $\RR^d$, and let $\mathring{A}=\{\mathring a_1,\dots,\mathring a_n\}$ be its image under the inverse stereographic projection, $\mathring{A}=\sterproj^{-1}(A)$. 

Then 
\begin{enumerate}[(i)]
 \item\label{it:spspheres} $\mathring a_i $ is above (resp. on, below) the hyperplane spanned by $\{\mathring a_{j_1},\dots, \mathring a_{j_{d+1}}\}$ if and only if $a_i$ is outside (resp. on, inside) the circumsphere spanned by $\{ a_{j_1},\dots, a_{j_{d+1}}\}$; and
 \item\label{it:sphyperplanes} for every hyperplane $H\subset \RR^d$, there is a hyperplane $\mathring H\subset \RR^{d+1}$  with $\NP\in\mathring H$ such that $\mathring a_i$ in $\mathring H$ (resp. $\mathring H^\pm$) if and only if $a_i$ in $H$ (resp. $H^\pm$).
\end{enumerate}
\end{lem}

\begin{lem}\label{lem:stackedinscribableisdelaunay}
Let $T$ be a $d$-dimensional polytopal subdivision with $n$ vertices whose
boundary is a $d$-simplex $P$,
and let $Q$ be polytopal complex (homeomorphic to a sphere) obtained by 
adding to $T$ the cones with apex $a_{n+1}$ over the faces of $P$.
Then the stereographic projection from $a_{n+1}$ induces a homeomorphism
\[\rsins(Q)\cong \rsdel(T).\]
\end{lem}
\begin{proof}
Let $\mr{A}=\{\mr a_1,\dots,\mr a_n,\mr a_{n+1}\} \subset \Sp{d}$ be an 
inscribed realization of $Q$. By a M\"obius transformation, we can assume that 
the last point lies at the north pole, 
$\mr a_{n+1}=\NP$. Now, by Lemma~\ref{lem:stereographic}, the Delaunay 
subdivision of the stereographic projection of the points $\mr a_ i$, 
$1\leq i\leq n$, coincides with $T$. Indeed, if $\mr S\subset \mr A$ is the set 
of vertices of a facet $F$ of $\conv(\mr A\cup \NP)$ that does not contain $\NP$,
then $\mr S$ spans a supporting hyperplane that has all the remaining points above it 
(at the same side as $\NP$). According to Lemma~\ref{lem:stereographic}\eqref{it:spspheres}, its 
stereographic projection $S=\sterproj(\mr S)$ spans an empty circumsphere,
and hence is the set of vertices of a cell of the Delaunay subdivision of $A$.
Additionaly, by Lemma~\ref{lem:stereographic}\eqref{it:sphyperplanes}
facets of $\conv(\mr A\cup \NP)$ that contain $\NP$ are in bijection 
with facets of $\conv(A)$, which by hypothesis is a simplex in any realization of $T$.

Moreover, every M\"obius transformation of $\Sp{d}$ that fixes the north pole 
induces a similarity of $\RR^d$. To conclude the proof,
observe that every realization of $T$ as Delaunay triangulation can be lifted 
with the inverse stereographic projection to a unique inscribed realization of $Q$.
\end{proof}

\begin{rmk}\label{rmk:fixedboundary}
 Notice that the exactly the same proof shows a bijection between realization spaces
 of inscribed polytopes and realization spaces of Delaunay subdivisions 
 \emph{with prescribed boundary}. Indeed, Lemma~\ref{lem:stereographic}\eqref{it:sphyperplanes}
 implies that the vertex figure of $\NP$ is combinatorially equivalent to
 the convex hull of the Delaunay triangulation. Since a general triangulation (as a simplicial complex)
 does not prescribe the convex hull of the realization, we have to focus only in 
 those whose convex hull is a simplex.
\end{rmk}

\subsection{Lexicographic liftings}
A central tool for our construction are lexicographic liftings, which are a way to derive $(d+1)$-dimensional point configurations from $d$-dimensional point configurations. 

\begin{dfn}\label{def:lexicographiclifting}
A \defn{lexicographic lifting} of a point configuration $A=\{a_1,\dots,a_n\}\subset\RR^d$ (with respect to the order induced by the labels) with a sign vector $(s_1,\dots,s_n)\in\{+,-\}^{n}$ is a configuration $\wh A=\{\wh a_1,\dots,\wh a_n,\wh a_{n+1}\}$ of $n+1$ labeled points in~$\RR^{d+1}$ such that:
\begin{enumerate}[(i)]
\item for $1\leq i\leq d$, $\wh a_i=(a_i,0)\in\RR^{d+1}$,
 \item for $d+1\leq i\leq n$, the point $\wh a_i$ lies in the half-line that starts at $\wh a_{n+1}$ and goes through $(a_i,0)$,
 \item \label{it:hyperplanes} for $d+1\leq i\leq n$, and for every hyperplane $H$ spanned by $d+1$ points of $\left\{\wh a_{{1}},\dots,\wh a_{{i-1}}\right\}$, 
 the points $\wh a_{n+1}$ and $\wh a_i$ lie at the same side of $H$ when $s_i=+$, and at opposite sides if $s_i=-$.
\end{enumerate}

If $s_i=+$ for every $1\leq i \leq n$, the lexicographic lifting is called \defn{positive}.
\end{dfn}

\begin{figure}[ht]
\centering
\begin{subfigure}[b]{0.23\textwidth}
\includegraphics[width=.9\linewidth]{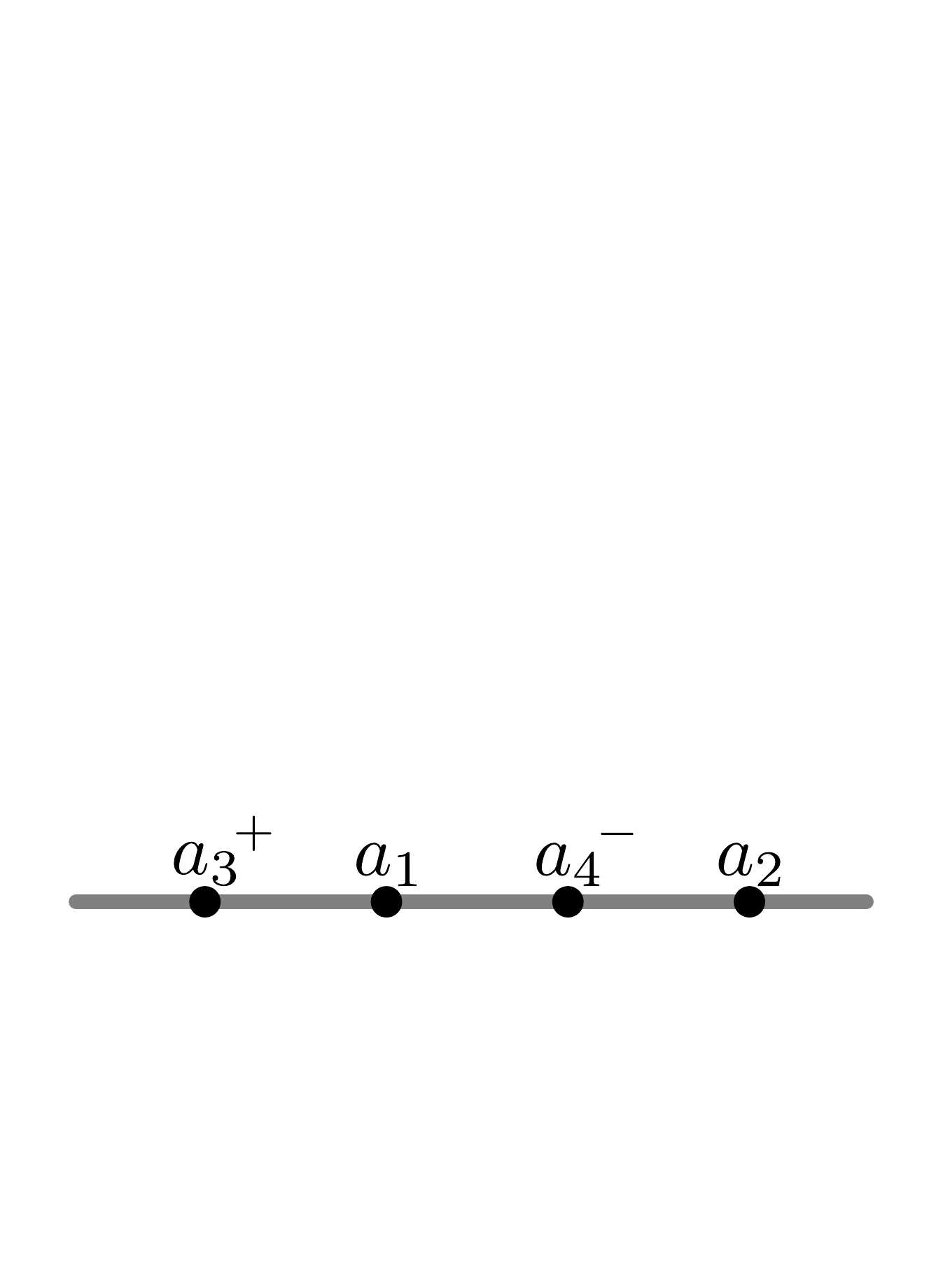}
\caption{$A$}
\label{fig:lifting1}
\end{subfigure}\quad
\begin{subfigure}[b]{0.23\textwidth}
\includegraphics[width=.9\linewidth]{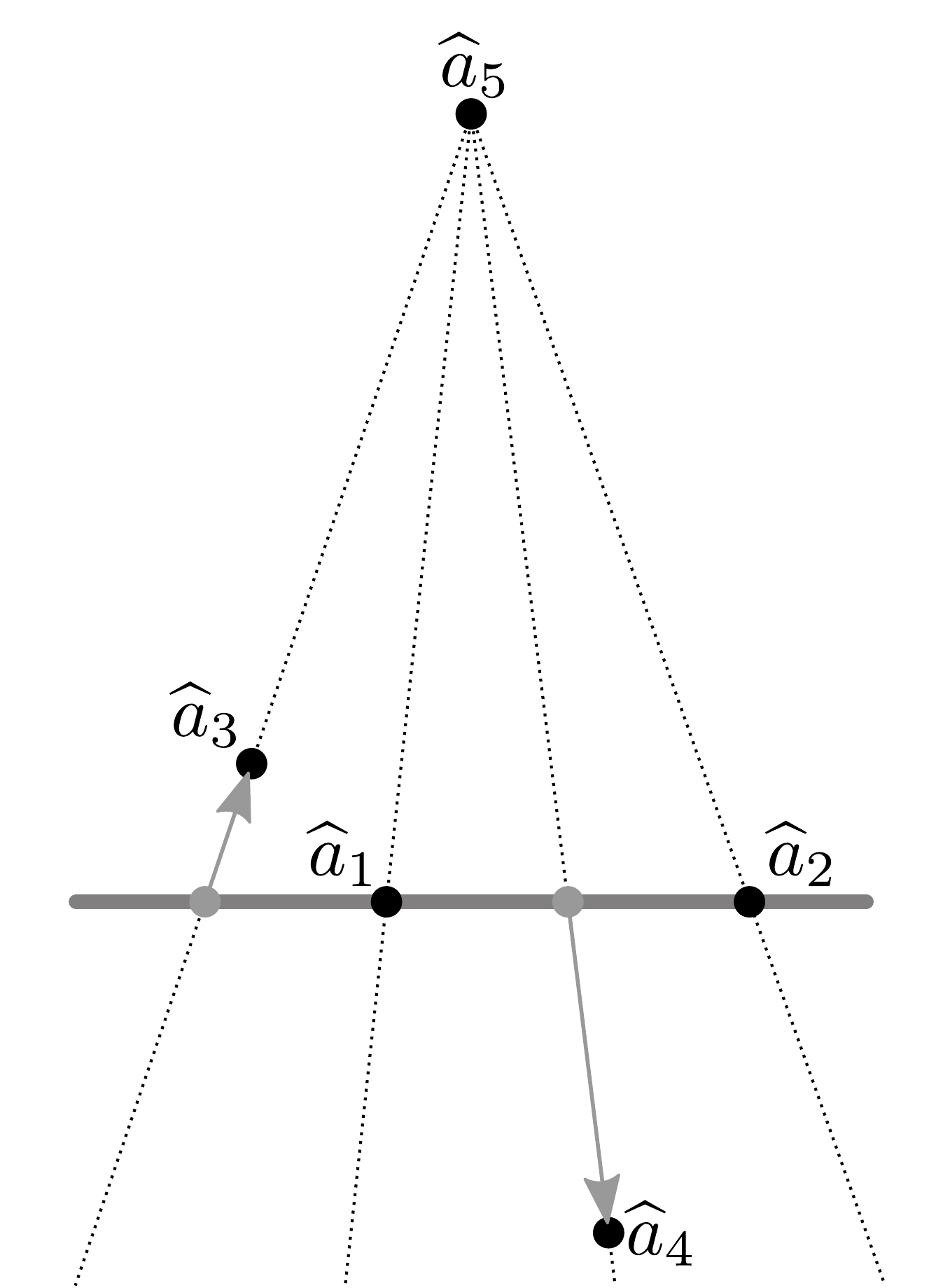}
\caption{$\wh A$}\label{fig:lifting2}
\end{subfigure}\quad
\begin{subfigure}[b]{0.23\textwidth}
\includegraphics[width=.9\linewidth]{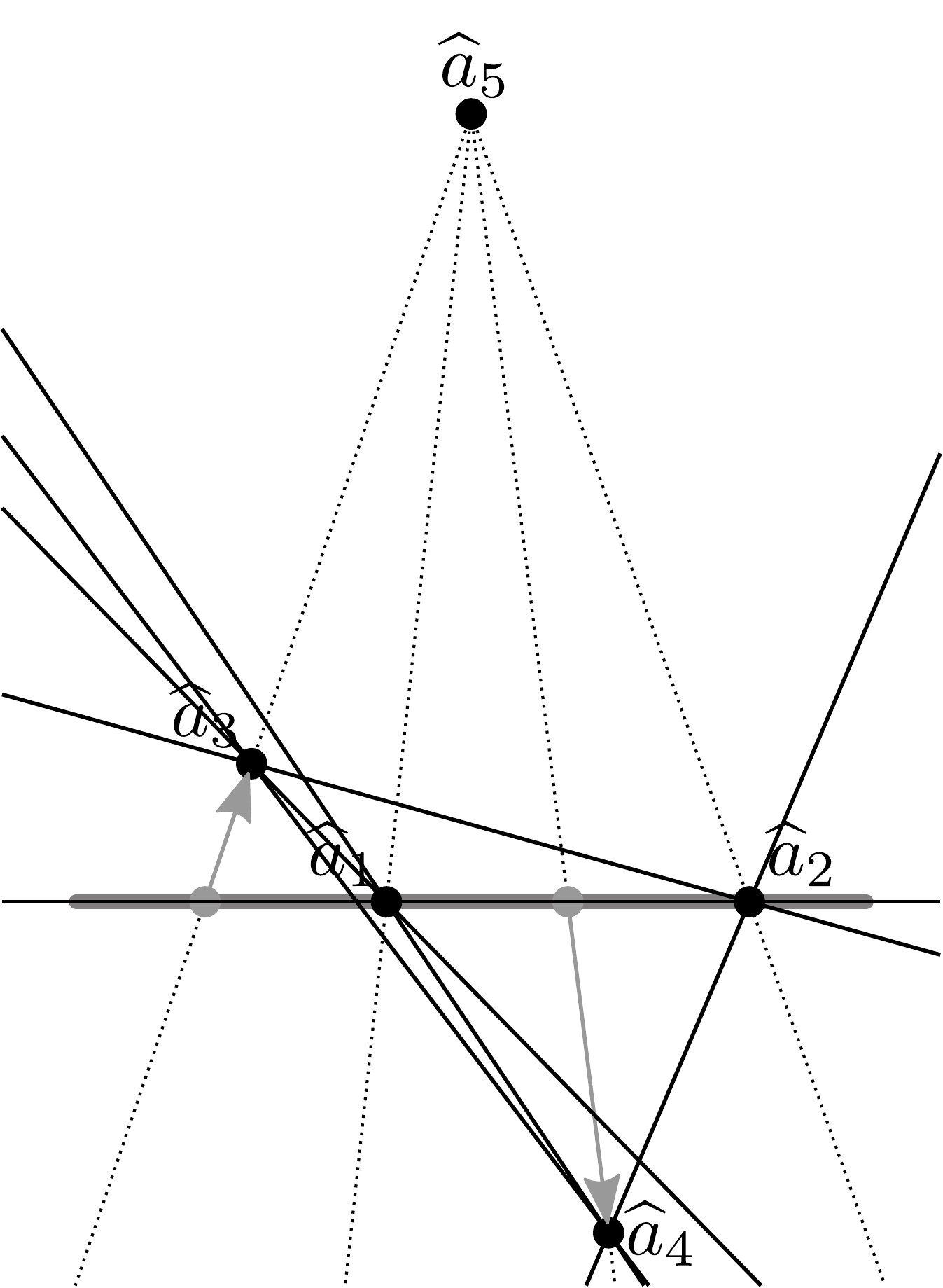}
\caption{Check (\ref{it:hyperplanes})}\label{fig:lifting3}
\end{subfigure}\quad
\begin{subfigure}[b]{0.23\textwidth}
\includegraphics[width=.9\linewidth]{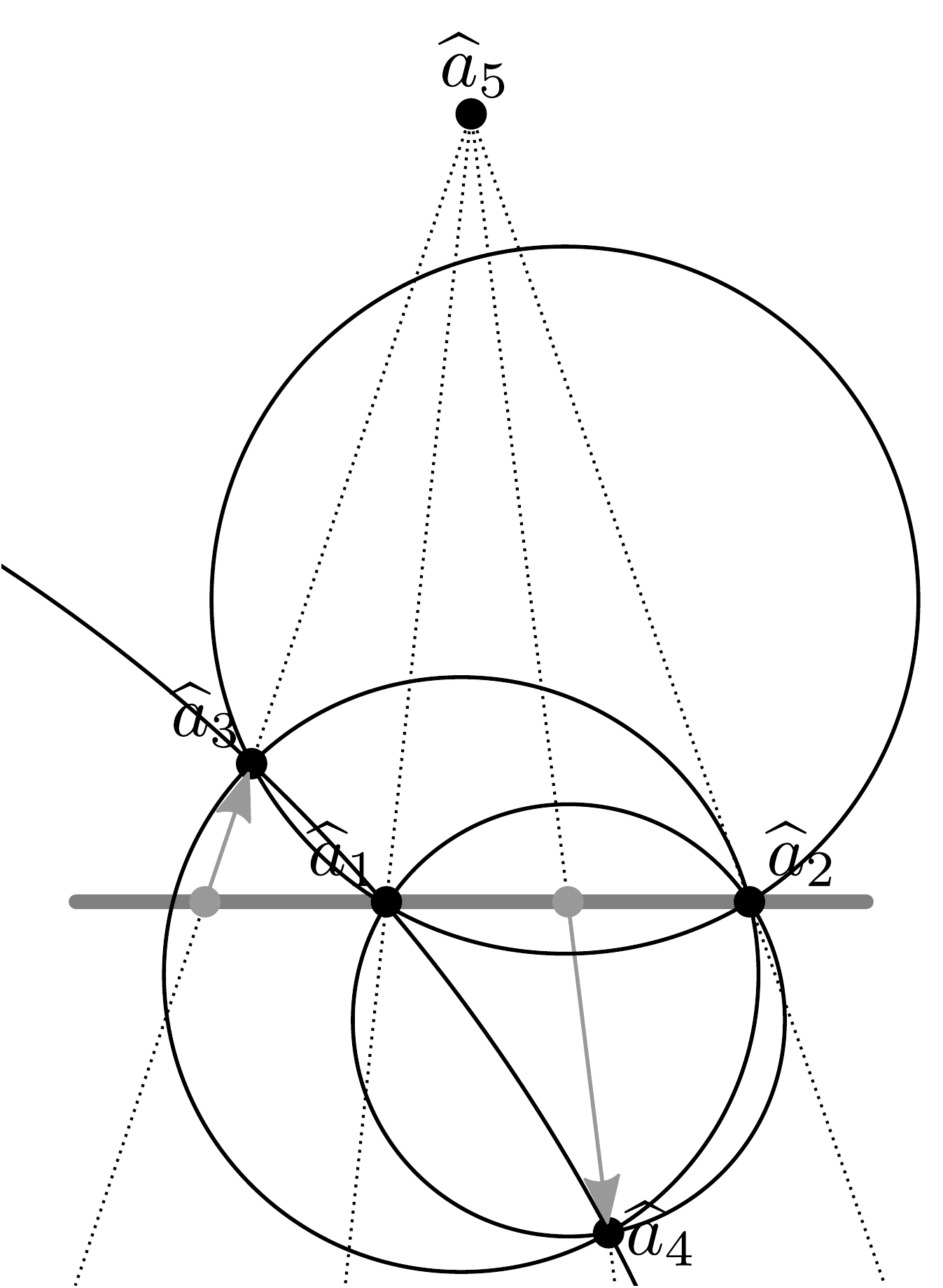}
\caption{Check Delaunay}\label{fig:lifting4}
\end{subfigure}
\caption{A Delaunay lexicographic lifting $\wh A\subset \RR^2$ of a configuration $A\subset \RR^1$. In \ref{fig:lifting3} one can check that (\ref{it:hyperplanes}) is fulfilled, and in \ref{fig:lifting4} that it is a Delaunay lexicographic lifting. This lifting is not positive.}
\label{fig:lifting}
\end{figure}

The proof of the following lemma is straightforward, since one can easily compute the chirotope of $\wh A$ from that of $A$ (compare \cite[Chapter~7]{OrientedMatroids1993}).
\begin{lem}
 The oriented matroid of a lexicographic lifting $\wh A$ of $A$ only depends on the oriented matroid of $A$ and the sequence of signs.
\end{lem}

An alternative way to see this is to observe that lexicographic liftings are dual to lexicographic extensions (cf. \cite[Section~7.2]{OrientedMatroids1993}).
\begin{rmk}\label{rmk:om}
A lex\-i\-co\-graph\-ic lifting of $A$ with signs $s_i$ realizes the dual oriented matroid of a lexicographic extension of the Gale dual of~$A$ with signature $[a_{n}^{-s_n},\dots,a_{d+2}^{-s_{d+2}}]$.
\end{rmk}

We use lexicographic liftings because they preserve homotopy of realization spaces 
(although for our proof we only need the surjectivity of the map $\rsom(\wh A)\rightarrow\rsom(A)$). 
This fact can be found in  {\cite[Lemma 8.2.1 and Proposition 8.2.2]{OrientedMatroids1993}}.
\begin{lem}\label{lem:liftingsurjective}
 For any lexicographic lifting $\wh A$ of $A$, $\rsom(\wh A)$ is homotopy equivalent to $\rsom(A)$.
\end{lem}
\begin{proof}[Proof (sketch)]
Any vector configuration $\wh V=\{\wh v_1,\dots, \wh v_{n+1}\}$ with the same oriented matroid as of 
$\homog{\wh A}$ can be mapped to a configuration $V=\{v_1,\dots,v_n\}$ that realizes $\homog{A}$, 
just by taking $v_i$ to be the orthogonal projection of $\wh v_i$ onto the hyperplane orthogonal to $\wh v_{n+1}$.
This defines a continuous map from $\rsom(\wh A)$ to $\rsom(A)$, which is easily seen to be surjective 
(compare Lemma~\ref{lem:surjective}). 

To see that this is indeed a homotopy equivalence, we can check that the fibers of this projection are balls.
Indeed, once the position of $\wh v_{n+1},\wh v_{n},\dots,\wh v_{i+1}$ is fixed,
the set of valid positions of $\wh v_i$ is a convex subset of the 
line that goes through $\wh v_{n+1}$ and $v_i$.
\end{proof}
To control their Delaunay triangulations, 
we use a particular family of lexicographic liftings (see also \cite{GonskaPadrol2013} and \cite{Seidel85}).
\begin{dfn}
 A \defn{Delaunay lexicographic lifting} of $A$ is a lexicographic lifting $\wh A$ such that for each $d+2<i\leq n+1$, $\wh a_{{i}}$ is not contained in any of the circumspheres of any simplex spanned by $d+2$ points of $\left\{\wh a_{{1}},\dots,\wh a_{{i-1}}\right\}$.
\end{dfn}

\begin{lem}\label{lem:surjective}
For any point configuration $A$ and any $s\in\{+,-\}^n$, there is Delaunay lexicographic lifting.  
\end{lem}
\begin{proof}
 To construct one, just replace $a_i$ by $\wh a_i=(a_i,h_i)\in\RR^{d+1}$ for some $h_i$ large enough so that $\wh a_{{i}}$ is above or below every hyperplane spanned by $\left\{\wh a_{{1}},\dots,\wh a_{{i-1}}\right\}$ and outside any of the circumspheres spanned by $\left\{\wh a_{{1}},\dots,\wh a_{{i-1}}\right\}$. Finally, set $\wh a_{{n+1}}$ to be the ``point at infinity'' $\wh a_{{n+1}}=(0,+\infty)$ and apply a projective transformation that preserves the hyperplane spanned by $\{a_1,\dots,a_d\}$ and sends $\wh a_{{n+1}}$ to $(0,h_{n+1})$ for some large $h_{n+1}>0$.  (This is possible, because 
the $\wh a_{{i}}$ are chosen so that the empty sphere condition will hold after a small perturbation.)
\end{proof}

We end with the following straightforward consequence of Lemma~\ref{lem:stereographic}.

\begin{cor}\label{cor:stereographicispositive}
Let $A=\{a_1,\dots,a_n\}$ be a configuration of $n$ labeled points in general position in~$\RR^d$ and let $\wh A$ be a Delaunay lexicographic lifting of $A$. 
Then $\sterproj^{-1}(\wh A)\cup \NP$ is a positive lexicographic lifting of $\wh A$ (with respect to the same order) inscribed on $\Sp{d+1}$.  
\end{cor}
\begin{proof}
The condition that $\wh A$ is a Delaunay lexicographic lifting implies that $a_j$ is outside every circumsphere spanned by points in $\{a_1,\dots,a_{j-1}\}$, which by Lemma~\ref{lem:stereographic} implies that $\mr a_j$ is above (i.e. at the same side as $\NP$) every hyperplane spanned by points in $\{\mr a_1,\dots,\mr a_{j-1}\}$.
\end{proof}

\subsection{Neighborly oriented matroids}

A crucial property of even-dimensional neighborly  configurations is that
their oriented matroids are \defn{rigid}. %

\begin{thm}[{\cite[Theorem 4.2]{Sturmfels1988}\cite{Shemer1982}}]\label{thm:rigid}
If $A$ is an even-dimensional neighborly point configuration, then the oriented matroid 
of $A$ is rigid, i.e. $\rsom(A)= \rspol(A)$.
\end{thm}
Kortenkamp \cite{Kortenkamp1997} found a way to use lexicographic liftings to construct neighborly 
point configurations.
\begin{thm}[{\cite[Theorem 1.2]{Kortenkamp1997}}]\label{thm:Kortenkamp}
For any point configuration $A$ with $d+4$ points in general position in $\RR^d$ there is an even-dimensional neighborly configuration $\wh A$ of $2d+8$ points in $\RR^{2d+4}$ obtained from $A$ by a sequence of lexicographic liftings.
\end{thm}
Finally, the following result can be found in \cite{Padrol2013} (compare also~\cite{GonskaPadrol2013}), where it is used to construct many neighborly polytopes.
\begin{thm}[{\cite[Theorem 4.2]{Padrol2013}}]\label{thm:neighborlylift}
 Let $A$ be a neighborly point configuration in general position, let $\wh A$ be a lexicographic lifting of $A$ and let $\wwh{A}$ be a positive lexicographic lifting of $\wh A$ (with respect to the same order). Then $\wwh A$ is neighborly.
\end{thm}

\subsection{The construction}
Here is the main technical result of this section.
\begin{lem}\label{lem:mainlifting}
For every configuration $A$ of $n$ points in general position in $\RR^{n-4}$
there exists an inscribed neighborly polytope $P$ with $2n+2$ vertices in $\RR^{2n-2}$ 
and sets $X$ and $Y$, homotopy equivalent to $\rsins(P)$ and $\rsom(A)$ 
respectively, 
such that $Y$ is a retract of $X$.
\end{lem}

\begin{proof}
For convenience, set $d = n - 4$.  Since $A$ is a $d$-dimensional configuration 
of $d + 4$ points, 
we can apply 
 the 
sequence of lexicographic liftings of Theorem~\ref{thm:Kortenkamp} to obtain a 
neighborly configuration $A_2$ of $2n$ points in general position 
in~$\RR^{2n-4}$. This configuration is obtained by lexicographic liftings and 
hence the corresponding realization spaces are homotopy equivalent, 
$\rsom(A) \htpy \rsom(A_2)$, by Lemma~\ref{lem:liftingsurjective}.
 
Now, we can apply a lexicographic lifting and a positive lexicographic lifting 
successively to obtain~$A_3=\wwh{A}\!\!{}_2$, which is a configuration of $N=2n+2$ points in 
general position in $\RR^{D}$, where $D=2n-2$. 
The convex hull of $A_3$ is a neighborly polytope $P$
by Theorem~\ref{thm:neighborlylift}. 
We will build a continuous surjection from $\rsins(P)\times \RR_{>0}^{N-1}$ onto $\rsom(A_2)$.

Let $B\subset\RR^{N\times D}$ be an inscribed realization of $P$, which is even-dimensional and 
neighborly. By Theorem~\ref{thm:rigid}, its oriented matroid is rigid, and hence 
the matroid of the vertices of $P$ coincides with the matroid of~$A_3$. Therefore,
the stereographic projection $\sterproj(B)$ of $B$ from $a_N$ is always a realization of $\wh A_2$.
Consider then the map $\varphi:\rsins(P)\times \RR_{>0}^{N-1}\to \rsom(\wh A_2)$ that maps $(B,\lambda)$ onto 
the configuration of vectors $\{(\lambda_i\sterproj(B_i),\lambda_i)\}_{1\leq i\leq N-1}\in \rsom(\wh A_2)\subseteq  \RR^{(N-1)\times D} $. The map is well defined,
because M\"obius transformations of $B$ induce similarities of $\sterproj(B)$, which
are affine transformations and hence induce linear transformations on $\varphi(B,\lambda)$.

Now we can use the projection map $\psi:\rsom(\wh A_3)\to \rsom(A_3) $ of Lemma~\ref{lem:liftingsurjective} to obtain a
realization of $A_3$.
From Corollary~\ref{cor:stereographicispositive} we deduce that the composition map $\psi\circ \varphi:\rsins(P)\times \RR_{>0}^{N-1}\to \rsom(A_3)$ is 
surjective. To conclude that $\psi\circ \varphi$ is a retraction, we have to exhibit a continuous inverse injection. 
But the construction of Corollary~\ref{cor:stereographicispositive} can easily be
performed in a continuous way. For example, one can use Hadamard's determinant inequalities
to find continuous heights that fulfill the constraints of Delaunay lexicographic liftings.
\end{proof}

 \begin{figure}[ht]
\centering
\includegraphics[width=.35\linewidth]{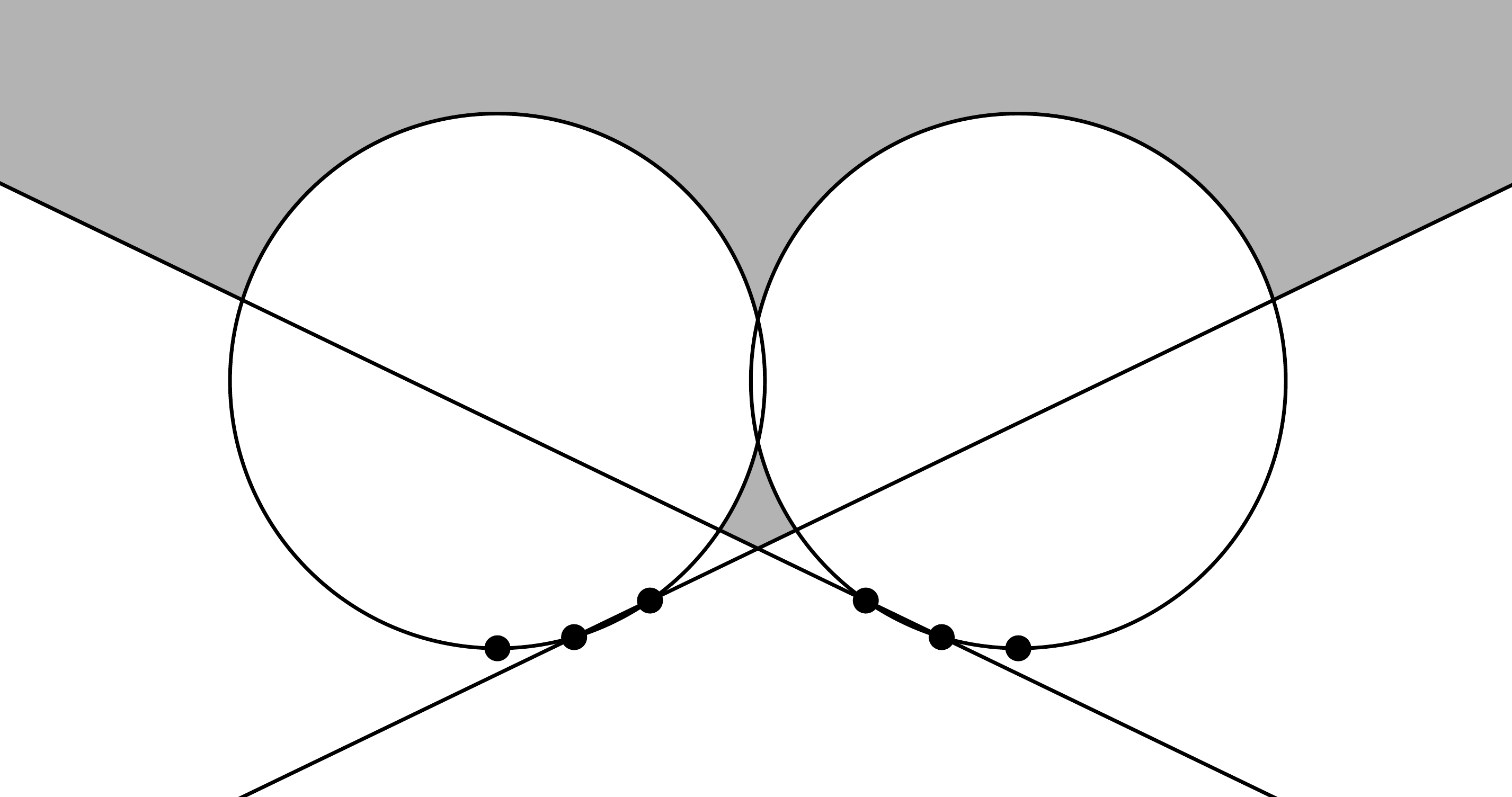}
\caption{Any point in the shaded area gives rise to the same Delaunay triangulation.}
\label{fig:disconnected}
\end{figure}
The reason why we cannot strengthen the statement to homotopy equivalence between 
$\rsins(P)$ and $\rsom(A)$, is that we do not understand 
the fibers of 
the map $\rsins(P)\twoheadrightarrow \rsom(A_3)$. We can prove that they 
are non-empty with Corollary~\ref{cor:stereographicispositive} but we cannot control 
their topology.  (Figure~\ref{fig:disconnected} shows an example of how disconnected
fibers might arise.)

\begin{thm}\label{thm:neighborly}
For every open primary basic semi-algebraic set $S$ there is a (neighborly) Delaunay 
triangulation and an inscribed simplicial (neighborly) polytope 
such that $S$ is a retract of their realization spaces, up to homotopy equivalence.
\end{thm}

\begin{proof}
A straightforward consequence of the Universality Theorem~\ref{thm:universality} is that
realization spaces of oriented matroids of configurations of $d+4$ points in $\RR^{d}$ 
exhibit universality. In particular, for every open primary basic semi-algebraic set $S$ there
is a configuration $A$ of $d+4$ points in general position in $\RR^{d}$ whose realization space
is homotopy equivalent to $S$. The proof is direct using oriented matroid duality (see \cite[Chapter~8]{OrientedMatroids1993})
after reorienting some elements (compare \cite[Corollary 6.16]{Ziegler1995}).

Hence, by Lemma~\ref{lem:mainlifting}, there is an inscribed simplicial neighborly $d$-polytope $P$ whose
realization space admits a continuous surjection onto a set homotopy equivalent to $S$.

For the claim concerning Delaunay triangulations, we consider the polytope $P'$
obtained by stacking a vertex on the facet $F=\{a_1,\dots,a_d\}$ of $P$. (The 
face lattice of $P'$ coincides with that of $P$, except that $F$ is replaced with its
stellar subdivision.)

We claim that every realization of $A$ can be lifted to an inscribed realization of~
$P'$ (and by construction, every realization of $P'$ can be projected to a realization of $A$). 
Indeed, to the configuration $A_2$ of Lemma~\ref{lem:mainlifting}, add a point
$a_0$ in the relative interior of $\conv(a_1,\dots,a_{d-1})$ and then apply a positive
Delaunay lexicographic lifting with order $a_1,\dots,a_{d-1},a_0,a_d,a_{d+1},\dots$. 
The Delaunay triangulation of this configuration clearly contains the stellar subdivision
of the simplex $\{a_1,\dots,a_d\}$.

An application of Lemma~\ref{lem:stackedinscribableisdelaunay} then concludes the proof.
\end{proof}

\subsection{Complexity}

A closer look into the proof of Lemma~\ref{lem:mainlifting} shows that all the operations
that we use are at the oriented matroid level 
(i.e., can be also applied to non-realizable matroids) and take only polynomial time. 
Therefore, for each rank~$3$ oriented matroid $M$ we can construct a (combinatorial)
Delaunay triangulation that is realizable \emph{if and only if} $M$ is.
An important consequence of the Universality Theorem
is Corollary~\ref{cor:omhard}, which states that realizability of rank~$3$ oriented matroids is polynomially equivalent to 
ETR \cite{Mnev1988}\cite{Shor1991}. Lemma~\ref{lem:mainlifting} implies that realizability of Delaunay triangulations is equally hard.

\begin{cor}\label{cor:hard}
The {realizability problem for Delaunay triangulations and simplicial inscribed polytopes} 
is polynomially equivalent to the \emph{existential theory of the reals (ETR)}. 
\end{cor}

Another consequence of the universality theorem for Delaunay triangulations is that
realization spaces can have an exponential number of connected components.

\begin{cor}\label{cor:asymptotic}
For every $m\geq 1$ there exist configurations of $O(m)$ points in general position in $\RR^{O(m)}$ whose 
realization spaces as Delaunay triangulations have at least $2^m$ connected components.
\end{cor}
\begin{proof}
Consider the polynomial $f_m(x)$ obtained recursively as follows:
 \begin{align*}
  f_0(x)=x^2-2,&&f_{ k+1}(x)=f_ {k}(f_0(x)).
 \end{align*}
That is, $f_1(x)=(x^2-2)^2-2$, $f_2(x)=((x^2-2)^2-2)^2-2$, $f_3(x)=(((x^2-2)^2-2)^2-2)^2-2$, and so on.
It is not hard to check that $f_m(x)$ has $2^{m+1}$ distinct simple real roots and that its arithmetic 
complexity is $O(m)$. 
The semi-algebraic set of points fulfilling $f_m(x)>0$ has at least $2^{m}$ connected components. 

Our claim now follows by the Universality Theorem~\ref{thm:universality} and Lemma~\ref{lem:mainlifting}.
\end{proof}

As a final remark in this section, we provide our smallest example of a Delaunay triangulation with disconnected realization space. 
It can constructed by applying Lemma~\ref{lem:mainlifting},  together with the stacking technique of the proof Theorem~\ref{thm:neighborly}, to the uniform rank~$3$ oriented matroid with $14$ elements found by Suvorov in 1988 \cite{Suvorov1988} (see also \cite[Chapter~8]{OrientedMatroids1993}), which has a disconnected  realization space. 

\begin{cor}\label{cor:small-config}
There is a $25$-dimensional configuration of $30$ points whose Delaunay triangulation has 
a disconnected realization space.
\end{cor}

%% file: Universality.bbl
\newcommand{\etalchar}[1]{$^{#1}$}
\providecommand{\noopsort}[1]{}
\providecommand{\bysame}{\leavevmode\hbox to3em{\hrulefill}\thinspace}
\providecommand{\MR}{\relax\ifhmode\unskip\space\fi MR }
\providecommand{\MRhref}[2]{%
  \href{http://www.ams.org/mathscinet-getitem?mr=#1}{#2}
}
\providecommand{\href}[2]{#2}
\begin{thebibliography}{JMLSW89}

\bibitem[AP14]{AdiprasitoPadrol2014}
Karim~A. Adiprasito and Arnau Padrol, \emph{{The universality theorem for
  neighborly polytopes}}, Preprint,
  \href{http://arxiv.org/abs/1402.7207}{arXiv:1402.7207}, February 2014.

\bibitem[AZ14]{AdiprasitoZiegler2014}
K.~A. Adiprasito and G.~M. Ziegler, \emph{{Many polytopes with low-dimensional
  realization space}}, Inventiones Mathematicae (2014), In press. Preprint
  available at \href{http://arxiv.org/abs/1212.5812v2}{arXiv:1212.5812v2}.

\bibitem[BCS97]{BuergisserClausenShokollahi1997}
Peter {B{\"u}rgisser}, Michael {Clausen}, and M.~Amin {Shokrollahi},
  \emph{{Algebraic complexity theory}}, Berlin: Springer, 1997, With the
  collaboration of Thomas Lickteig.

\bibitem[BLS{\etalchar{+}}93]{OrientedMatroids1993}
Anders Bj{\"o}rner, Michel {Las Vergnas}, Bernd Sturmfels, Neil White, and
  G{\"u}nter~M. Ziegler, \emph{{Oriented matroids.}}, {Encyclopedia of
  Mathematics and Its Applications. 46. Cambridge: Cambridge University Press.
  516 p. }, 1993 (English).

\bibitem[Bro79]{Brown1979}
Kevin~Q. Brown, \emph{{Voronoi diagrams from convex hulls.}}, Inf. Process.
  Lett. \textbf{9} (1979), 223--228.

\bibitem[{Ede}06]{Edelsbrunner2006}
Herbert {Edelsbrunner}, \emph{{Geometry and topology for mesh generation. 1st
  paperback ed.}}, 1st paperback ed. ed., Cambridge: Cambridge University
  Press, 2006.

\bibitem[FG11]{FuterGueritaud11}
David Futer and Fran\c{c}ois Gu{\'e}ritaud, \emph{{From angled triangulations
  to hyperbolic structures}}, {Interactions between hyperbolic geometry,
  quantum topology and number theory}, {Contemp. Math.}, vol. 541, Amer. Math.
  Soc., Providence, RI, 2011, pp.~159--182. \MR{2796632 (2012j:57038)}

\bibitem[For97]{Fortune1997}
Steven Fortune, \emph{{Voronoi diagrams and {D}elaunay triangulations}},
  {Handbook of discrete and computational geometry}, {CRC Press Ser. Discrete
  Math. Appl.}, CRC, Boca Raton, FL, 1997, pp.~377--388.

\bibitem[GO04]{HanbookDCG}
Jacob~E. Goodman and Joseph O'Rourke (eds.), \emph{{Handbook of discrete and
  computational geometry}}, second ed., {Discrete Mathematics and its
  Applications (Boca Raton)}, Chapman \& Hall/CRC, Boca Raton, FL, 2004.

\bibitem[GP13]{GonskaPadrol2013}
Bernd Gonska and Arnau Padrol, \emph{{Neighborly inscribed polytopes and
  {D}elaunay triangulations}}, Preprint,
  \href{http://arxiv.org/abs/1308.5798}{arXiv:1308.5798}, Aug 2013.

\bibitem[GZ11]{GonskaZiegler2011}
Bernd Gonska and G{\"u}nter~M Ziegler, \emph{{Inscribable stacked polytopes}},
  Preprint, \href{http://arxiv.org/abs/1111.5322}{arXiv:1111.5322}, Nov 2011.

\bibitem[HRGZ97]{HenkRichterGebertZiegler1997}
Martin Henk, J{\"u}rgen Richter-Gebert, and G{\"u}nter~M. Ziegler, \emph{{Basic
  properties of convex polytopes}}, {Handbook of discrete and computational
  geometry}, {CRC Press Ser. Discrete Math. Appl.}, CRC, Boca Raton, FL, 1997,
  pp.~243--270.

\bibitem[JMLSW89]{JaggiManiLevitskaSturmfelsWhite1989}
Beat Jaggi, Peter Mani-Levitska, Bernd Sturmfels, and Neil White,
  \emph{{Uniform oriented matroids without the isotopy property.}}, {Discrete
  Comput. Geom.} \textbf{4} (1989), no.~2, 97--100.

\bibitem[Kor97]{Kortenkamp1997}
Ulrich~H. Kortenkamp, \emph{{Every simplicial polytope with at most $d+4$
  vertices is a quotient of a neighborly polytope.}}, Discrete Comput. Geom.
  \textbf{18} (1997), no.~4, 455--462.

\bibitem[Mn{\"e}88]{Mnev1988}
Nikolai~E. Mn{\"e}v, \emph{{The universality theorems on the classification
  problem of configuration varieties and convex polytopes varieties}},
  {Topology and geometry---{R}ohlin {S}eminar}, {Lecture Notes in Math.}, vol.
  1346, Springer-Verlag, Berlin Heidelberg, 1988, pp.~527--544.

\bibitem[Mun75]{Munkres}
James~R. Munkres, \emph{{Topology: a first course}}, Prentice-Hall, Inc.,
  Englewood Cliffs, N.J., 1975.

\bibitem[Pad13]{Padrol2013}
Arnau Padrol, \emph{{Many neighborly polytopes and oriented matroids.}},
  {Discrete Comput. Geom.} \textbf{50} (2013), no.~4, 865--902.

\bibitem[PT14]{PadrolTheran14}
Arnau Padrol and Louis Theran, \emph{Delaunay triangulations with disconnected
  realization spaces}, Symposium on Computational Geometry (Siu-Wing Cheng and
  Olivier Devillers, eds.), ACM, 2014, p.~163.

\bibitem[RG95]{RichterGebert95}
J{\"u}rgen Richter-Gebert, \emph{{Mn{\"e}v's universality theorem revisited}},
  S{\'e}m. Lothar. Combin. \textbf{34} (1995), Art.\ B34h, approx.\ 15 pp.\
  (electronic).

\bibitem[RG96]{RichterGebert1997}
\bysame, \emph{{Realization spaces of polytopes}}, {Lecture Notes in
  Mathematics}, vol. 1643, Springer-Verlag, Berlin, 1996.

\bibitem[RG98]{RichterGebert1998}
\bysame, \emph{{The universality theorems for oriented matroids and
  polytopes}}, {Advances in Discrete and Computational Geometry (Mount Holyoke
  1996)} (B.~Chazelle, J.~E. Goodman, and R.~Pollack, eds.), {Contemporary
  Mathematics}, vol. 223, Amer. Math. Soc., Providence RI, 1998, pp.~269--292.

\bibitem[RGZ97]{RichterGebertZiegler1997}
J{\"u}rgen Richter-Gebert and G{\"u}nter~M. Ziegler, \emph{{Oriented
  matroids}}, {Handbook of discrete and computational geometry}, {CRC Press
  Ser. Discrete Math. Appl.}, CRC, Boca Raton, FL, 1997, pp.~111--132.

\bibitem[{Ric}96]{RichterGebert1996}
J{\"u}rgen {Richter-Gebert}, \emph{{Two interesting oriented matroids.}}, {Doc.
  Math., J. DMV} \textbf{1} (1996), 137--148.

\bibitem[Riv94]{Rivin1994}
Igor Rivin, \emph{{Euclidean structures on simplicial surfaces and hyperbolic
  volume.}}, {Ann. Math. (2)} \textbf{139} (1994), no.~3, 553--580.

\bibitem[Riv96]{Rivin1996}
\bysame, \emph{{A characterization of ideal polyhedra in hyperbolic
  {$3$}-space}}, Ann. of Math. (2) \textbf{143} (1996), no.~1, 51--70.

\bibitem[Riv03]{Rivin2003}
\bysame, \emph{{Combinatorial optimization in geometry.}}, {Adv. Appl. Math.}
  \textbf{31} (2003), no.~1, 242--271.

\bibitem[Sei85]{Seidel85}
Raimund Seidel, \emph{{A method for proving lower bounds for certain geometric
  problems}}, {Computational Geometry} (G.~T. Toussaint, ed.), North-Holland,
  Amsterdam, Netherlands, 1985, pp.~319--334.

\bibitem[She82]{Shemer1982}
Ido Shemer, \emph{{Neighborly polytopes.}}, Isr. J. Math. \textbf{43} (1982),
  291--314.

\bibitem[Sho91]{Shor1991}
Peter~W. Shor, \emph{{Stretchability of pseudolines is {NP}-hard}}, {Applied
  geometry and discrete mathematics}, {DIMACS Ser. Discrete Math. Theoret.
  Comput. Sci.}, vol.~4, Amer. Math. Soc., Providence, RI, 1991, pp.~531--554.

\bibitem[Ste32]{Steiner1832}
Jacob Steiner, \emph{{\em {S}ystematische {E}ntwicklung der {A}bh{\"a}ngigkeit
  geometrischer {G}estalten von einander}}, {Fincke, Berlin}, 1832, Also in:
  Gesammelte Werke, Vol.~1, Reimer, Berlin 1881, pp. 229--458.

\bibitem[Ste28]{Steinitz1928}
Ernst Steinitz, \emph{{{\"U}ber isoperimetrische Probleme bei konvexen
  Polyedern.}}, J. f. M. \textbf{159} (1928), 133--143 (German).

\bibitem[Stu88]{Sturmfels1988}
Bernd Sturmfels, \emph{{Neighborly polytopes and oriented matroids.}}, Eur. J.
  Comb. \textbf{9} (1988), no.~6, 537--546.

\bibitem[Suv88]{Suvorov1988}
P.~Suvorov, \emph{{Isotopic but not rigidly isotopic plane systems of straight
  lines}}, {Topology and geometry---{R}ohlin {S}eminar}, {Lecture Notes in
  Math.}, vol. 1346, Springer-Verlag, Berlin Heidelberg, 1988, pp.~545--556.

\bibitem[{Tsu}13]{Tsukamoto2013}
Yasuyuki {Tsukamoto}, \emph{{New examples of oriented matroids with
  disconnected realization spaces.}}, {Discrete Comput. Geom.} \textbf{49}
  (2013), no.~2, 287--295.

\bibitem[Vak06]{Vakil06}
Ravi Vakil, \emph{{Murphy's law in algebraic geometry: badly-behaved
  deformation spaces}}, Invent. Math. \textbf{164} (2006), no.~3, 569--590.

\bibitem[Val79]{Valiant1979}
Leslie~G. Valiant, \emph{{Completeness classes in algebra}}, {Conference
  {R}ecord of the {E}leventh {A}nnual {ACM} {S}ymposium on {T}heory of
  {C}omputing ({A}tlanta, {G}a., 1979)}, ACM, New York, 1979, pp.~249--261.

\bibitem[Ver88]{Vershik1988}
Anatoly~M. Vershik, \emph{{Topology of the convex polytopes{\rq} manifolds, the
  manifold of the projective configurations of a given combinatorial type and
  representations of lattices}}, {Topology and geometry---{R}ohlin {S}eminar},
  {Lecture Notes in Math.}, vol. 1346, Springer-Verlag, Berlin Heidelberg,
  1988, pp.~557--581.

\bibitem[{Whi}89]{White1989}
Neil~L. {White}, \emph{{A nonuniform matroid which violates the isotopy
  conjecture}}, {Discrete Comput. Geom.} \textbf{4} (1989), no.~1, 1--2.

\bibitem[Zie95]{Ziegler1995}
G{\"u}nter~M. Ziegler, \emph{{Lectures on polytopes}}, {Graduate Texts in
  Mathematics}, vol. 152, Springer-Verlag, New York, 1995.

\end{thebibliography}
